\documentclass[11pt]{article}

\usepackage{amssymb}
\usepackage{amsmath,amsthm}
\usepackage{graphicx,fullpage}

\usepackage[ruled,linesnumbered]{algorithm2e}

\usepackage[shortlabels]{enumitem}
\setlist{
  listparindent=3em,
  itemindent=\parindent,
  parsep=0pt,
  leftmargin=\parindent
}

\newlist{nested}{enumerate}{5}
\setlist[nested]{
  nosep,    
  noitemsep,
  listparindent=2\parindent,
    leftmargin=2\parindent,
  parsep=0pt
}

\usepackage{caption}   
\usepackage{color,soul}
\usepackage[table]{xcolor}

\usepackage{hyperref}
\PassOptionsToPackage{unicode}{hyperref}

\newtheorem{theorem}{Theorem}

 \newtheorem{cor}[theorem]{Corollary}
 
 \newtheorem{prop}[theorem]{Proposition}

\newcounter{cases}
\newcounter{subcases}[cases]

\def\F{\mathbb{F}}

\def\Z{\mathbb{Z}}

\newcommand{\ie} [1] {\textit{i.e.,} #1}

\newcommand{\ceil}[1]{\lceil #1\rceil}
\newcommand{\floor}[1]{\lfloor #1\rfloor}

\def\id{\varepsilon}

\begin{document}

\date{}

\title{Bounds for Permutation Arrays under Kendall Tau Metric$^*$} 

\author{Sergey Bereg
\and William Bumpass
\and Mohammadreza Haghpanah
\and Brian Malouf
\and I. Hal Sudborough
}


\maketitle

\begin{abstract}Permutation arrays under the Kendall-$\tau$ metric have been considered for error-correcting codes.
Given $n$ and $d\in [1..\binom{n}{2}]$, the task is to find a large permutation array of permutations on $n$ 
symbols with pairwise Kendall-$\tau$ distance at least $d$. 
Let $P(n,d)$ denote the maximum size of any permutation array of permutations on $n$ symbols with pairwise Kendall-$\tau$ distance $d$. New algorithms and several theorems are presented, giving improved lower bounds for $P(n,d)$. Also, $(n,m,d)$-arrays are defined, which are permutation arrays on n symbols with Kendall-$\tau$ distance d, with the restriction that symbols {1...(n-m)} appear in increasing order. Let $P(n,m,d)$ denote the maximum size of any $(n,m,d)$-array.
For example, (n,m,d)-arrays are useful for recursively computing lower bounds for $P(n,d)$. Lower and upper bounds are given for $P(n.m,d)$.
\end{abstract}


\section{Introduction}
\label{sec:intro}

In \cite{Abdollahi22,bm-cpec-10,buzaglo,Jiang,Vijayakumaran,wang17},  permutation arrays under the Kendall-$\tau$ metric were studied. This complemented many studies of permutation arrays under other metrics, such as the Hamming metric  \cite{bls-18} \cite{bmmms-19}
\cite{chu2004}, Chebyshev metric \cite{Klove10}  and several others \cite{Deza2}.
The use of the Kendall-$\tau$ metric was motivated by applications of error correcting codes and rank modulation in flash memories \cite{Jiang}. 

Let $[1..n]=\{1,2,\dots,n\}$ and let $S_n$ denote the set of all permutations over $[1..n]$.
Let $\sigma$ and $\pi$ be two permutations (or strings\footnote{ If $\sigma$ and $\tau$ are strings then every symbol must appear the same number of times in them.}) over an alphabet $\Sigma \subseteq [1..n]$.
The {\em Kendall-$\tau$ distance} between $\sigma$ and $\pi$, denoted by $d(\sigma,\pi)$,
is the minimum number of adjacent transpositions (bubble sort operations) required to transform $\sigma$ into $\pi$. For an array (set) $A$ of permutations (strings), the pairwise Kendall-$\tau$ distance of $A$, denoted by $d(A)$, is $\min \{ d(\sigma,\pi)~ |~ \sigma, \pi \in A , \sigma \ne \pi \}$. 
An array $A$ of permutations on $[1..n]$ with $d(A)=d$ will be called a $(n,d)$-PA or simply an {\em $(n,d)$-array}. 
Let $P(n,d)$ denote the maximum cardinality of any $(n,d)$-PA $A$. 

Vijayakumaran \cite{Vijayakumaran} showed several lower bounds for $P(5,d)$ and $P(6,d)$ using integer linear programming. 
Buzaglo and Etzion \cite{buzaglo} showed many new bounds, including $P(7,3) \ge 588$ using two permutation representatives and a set of permutations generated by specific automorphism operations.
We also show results using automorphisms, namely those given in Table \ref{Automorphism}. Details of these automorphisms are given in Section 4 with additional details in an appendix.

We also used other programs to compute good lower bounds:
\begin{enumerate}
\item 
{\em Maximum Clique}. 
Let $G_{n,d}=(S_n,E)$ be a graph where two nodes are connected by an edge if the 
corresponding permutations are at Kendall-$\tau$ distance at least $d$. 
A clique in $G_{n,d}$ is an $(n,d)$-array. Compute $P(n,d)$ as the maximum size of a clique in $G_{n,d}$.
\item
{\em Random Greedy}. 
Choose a random $(n,d)$-array of small size (1-5). Proceed through all remaining permutations in lexicographic order and add them to the set if they have Kendall-$\tau$ distance at least $d$. 
\end{enumerate}

In Tables \ref{Automorphism}
and \ref{Random}
are given sporadic results obtained by these techniques. Blank positions in our tables signify other papers have the best lower bounds known {\it e.g.} \cite{buzaglo}, \cite{Vijayakumaran}. All other lower bounds we give are larger than the previous lower bounds, except for the two noted in Table \ref{Automorphism}.

\begin{table}[htb]
\centering
\begin{tabular}{|c|c|c|c|c|c|c|c|}
    \hline
    n:d & 3 &4&5&6&7&8&9\\ 
    \hline    \hline
    6& $102^{(*)}$&&&&&&\\
    7& $588^{(*)}$ &336&126&84&42&&\\
    8&3,752&2,240&672&448&168&&\\
    9&&&&&1,008&&288\\
    \hline
\end{tabular}
    \vspace{5pt}
    \caption{Improved lower bounds on $P(n,d)$ by automorphisms.
    (The bounds for $P(6,3)$ and $P(7,3)$ are from \cite{Vijayakumaran} and \cite{buzaglo}, respectively.)}
    \label{Automorphism}
\end{table}

\begin{table}[htb]
\centering
\begin{tabular}{|c|c|c|c|c|c|c|c|}
\hline
n:d & 3 &4&5&6&7&8&9\\ 
\hline    \hline
8&&&&&&115&57\\
9&26,831&15,492&3,882&2,497&see Table 1&608&see Table 1\\
10&233,421&133,251&29,145&18,344&5,629&3,832&1,489\\
11&&&247,014&153,260&42,013&28,008&9,747\\
\hline
\end{tabular}

\vspace{5mm}

\begin{tabular}{|c|c|c|c|c|c|c|c|c|c|}
\hline
n:d & 10&11&12&13&14&15&16&17&18|\\ 
\hline    \hline
7&13&8&7&4&&&&&\\
8&43&26&21&15&12&8&&&\\
9&195&101&79&46&37&24&19&15&15\\
10&1,070&491&371&196&153&90&71&&\\
11&6,890&2,861&2,108&1,007&773&415&329&191&131\\
12&50,649&19,227&13,935&6,087&4,564&2,250&1,730&936&740\\
\hline
\end{tabular}
\vspace{5pt}
\caption{Improved lower bounds by random Greedy.
    }
\label{Random}
\end{table}

In \cite{bm-cpec-10}  Barg and Mazumdar described their Theorem 4.5, which is given below:

\begin{theorem}
\label{Barg} 
\cite{bm-cpec-10} Let $m = ((n - 2)^{t+1} - 1)/(n - 3)$, where $n - 2$ is a prime power.  
 Then $$P(n,2t+1) \ge \frac{n!}{t(t+1)m}.$$
\end{theorem}

This was improved by Wang,  Zhang, Yang, and Ge in \cite{wang17}. 

\begin{theorem}
\label{Ge}
\cite{wang17} 
Let $m = ((n - 2)^{t+1}  -
  1)/(n - 3)$, where $n - 2$ is a prime power.  
Then 
\begin{equation}
\label{example2}
P(n,2t+1) \ge \frac{n!}{(2t+1)m}.                    
\end{equation}
\end{theorem}

For example, by choosing $t=1$ and $n=11$, one obtains, by Theorem \ref{Ge}, $P(11,3) \ge 1{,}330{,}560.$
Theorem \ref{Ge} applies only when $n$ is two greater than the power of a prime. 
To compute good lower bounds for $P(n,d)$ when $n$ is not two greater than the power of a prime, one needs other techniques (such as those given in Theorems \ref{Even}, \ref{Jiang2}, \ref{Jiang}, \ref{P_n_2_3}, \ref{improved}, \ref{gen}, and \ref{better}).
The lower bounds given by Theorem \ref{Ge} are close to corresponding upper bounds when t is small, but not so close when t is large, say larger than $n$). For example, observe that, in the right side of the inequality (1), $\frac{n!}{(2t+1)m} < 1$, when $t \ge n$, because $m = ((n - 2)^{t+1}  -
  1)/(n - 3)$, and $(n - 2)^{t+1} > n!$ when $t>n \ge 4$. On the other hand, $P(2^i \cdot 6,2^i \cdot 15) \ge 2^{2^i}$, for every integer $i \ge 0$. This follows from Theorem \ref{double}. 

\begin{theorem}
\label{double}
For all $n \ge 3$ and all $d \ge 2$, $P(2n,2d) \ge P(n,d)^2$. 
\end{theorem}

\begin{proof}
Let $A$ be a $(n,d)$-array of permutations on $[1..n]$. 
Let $O$ be the function from $[1..n]$ to $[1..2n]$ defined by $O(i) = 2i-1$, for all $i$, and let $E$ be the function from $[1..n]$ to $[1..2n]$ defined by $E(i) = 2i$, for all $i$. 
By a slight abuse of notation, let $O$ and $E$ also denote the corresponding string homomorphisms. 
Create the array $C = \{ c_1 c_2 ... c_{2n} = (O(\sigma),E(\tau)) ~|~ \sigma = a_1 a_2 ... a_n, \tau = b_1 b_2 ... b_n\in A$ and $c_1 c_3 ... c_{2n-1} = O(\sigma)$, and $c_2 c_4 ... c_{2n}= E(\tau)\}$.
Set $C$ is a $(2n,2d)$-array with $|A|^2$ elements. 
That is, for any permutation in C, odd (and only odd) index positions have an odd numbered symbol. 
 For any adjacent symbols, say $a_i$ and $a_{i+1}$ in $\sigma$ it takes at least two adjacent transpositions to exchange the corresponding symbols, $O(a_i)$ and $O(a_{i+1})$ in $(O(\sigma),E(\tau))$, for any $\tau$, due to the even numbered symbol between them. Changing the order of the even numbered symbols doesn't make it easier to rearrange the odd numbered symbols. 
 The argument is similar to the rearrangement of the even numbered symbols.
\end{proof}

Observe $P(6,15) = 2$, as a set with a permutation $\pi$ on six symbols and the reversal of $\pi$ has Kendall-$\tau$ distance 15. So, $P(2^i \cdot 6,2^i \cdot 15) \ge 2^{2^i}$, for all integer $i \ge 0$, follows from Theorem \ref{double}. As the lower bounds given by Theorem \ref{Ge}, in these cases, is less than one, Theorem \ref{Ge} can be significantly improved when t is large. Several additional examples of such improved lower bounds are shown in Tables \ref{better2a} and \ref{better2b}. Note that the recursive technique given in Theorem \ref{Jiang} below also doesn't provide good lower bounds in such cases, as $d > n$ implies $\ceil{\frac{n+1}{d}} = 1$.

\begin{cor}
\label{doublesquare}
For all $n \ge 3$ and $d \ge 2$ with $d \le n$, $P(2n,2d) \ge 2 \cdot P(n,d)^2$. 
\end{cor}

\begin{proof}
The proof is similar to that of Theorem \ref{double}. Create the array $D = \{ c_1 c_2 ... c_{2n} = (O(\sigma),E(\tau)),$

\noindent $d_1 d_2 ... d_{2n} = (E(\sigma),O(\tau)) ~|~ \sigma = a_1 a_2 ... a_n,~ \tau = b_1 b_2 ... b_n\in A$ and $c_1 c_3 ... c_{2n-1} = O(\sigma)$, and $c_2 c_4 ... c_{2n}= E(\tau),
d_1 d_3 ... d_{2n-1} = E(\sigma)$, and $d_2 d_4 ... d_{2n}= O(\tau)\}$.
As it takes $2n$ adjacent transpositions to transform an alternating odd-even string into an alternating even-odd string, or vice-versa, $D$ is a $(2n,2d)$-array with $2 \cdot |A|^2$ elements, when $d \le n$. 
\end{proof}

For example, we see (from Table \ref{better2b}),  $P(13,13) \ge 14,158$. 
By Corollary \ref{doublesquare},
$P(26,26) \ge 400,897,928$. By Theorem \ref{Ge},
$P(26,26) \ge 16,959,621$. Again, using Corollary \ref{doublesquare} gives a better lower bound.

\begin{cor} \label{cor5}
For all $n \ge 3$, all $d \ge 2$, and $s$>2, $P(sn,sd) \ge P(n,d)^s$. 
\end{cor}

\begin{proof}
The proof is similar to that of Theorem \ref{double}. 
Let $A$ be a $(n,d)$-array. 
Create a new set of permutations, say $Q$, so that in every permutation in $Q$, index $\equiv ~k \pmod s$  positions contain (and only contain) index $ \equiv ~k\pmod s$ symbols. Further, for the $i^{th}$ equivalence class  (mod s), the string of symbols in positions in the $i^{th}$ equivalence class (mod s) is a string in $A$ transformed by $T_{i}(j) = s \cdot (j-1) + i$. 
\end{proof}

For example, Corollary \ref{cor5} implies 
$P(3^i \cdot 6,3^i \cdot 15) \ge 2^{3^i}$, for all integer $i \ge 0$.
The following theorems from \cite{Jiang} allow one to obtain good lower bounds in many cases.

\begin{theorem}
\label{Even}
\cite{Jiang} 
For all $n \ge 1$ and even $d \ge 2$, we have $P(n,d) \ge \frac{1}{2}P(n,d-1)$.
\end{theorem}

\begin{theorem} 
\label{Jiang2}
\cite{Jiang} For all $n, d \ge 1$ we have
$P(n +1, d) \le (n+1) \cdot P(n, d)$. 
\end{theorem}

Using Theorems \ref{Jiang2} and \ref{Ge}
we have 
$P(14,11)\ge P(15,11)/15\ge 15!/(11\cdot 402234\cdot 15)\approx 19,703.2$

\begin{theorem} 
\label{Jiang}
\cite{Jiang} For all $n, d > 1$ we have
$P(n+1,d) \ge
\ceil{\frac{n+1}{d}}
P(n, d)$.
\end{theorem}

For example, to compute a lower bound for $P(14,11)$
one can use, iteratively, Theorem \ref{Jiang} to obtain $P(14,11) \ge \ceil{\frac{14}{11}}\cdot
\ceil{\frac{13}{11}}
\cdot P(12,11) = 4\cdot P(12,11)$.  
By computation  (using the random greedy algorithm) we have $P(12,11) \ge 19,277$, so  

\begin{equation}
\label{example1}
P(14,11) \ge 77,108.  
\end{equation}

Theorems \ref{gen} and \ref{better}, described in the next section, give a technique to compute better lower bounds, in cases where t is large, as summarized in Tables \ref{better2a} and \ref{better2b}.

\section{Recursive Techniques}

There is often a better way to compute $P(n,d)$ from $P(n-m,d)$, for various $m \ge 1$, than an iterative use of Theorem \ref{Jiang}.
Let $S_{n,m}$ be the 
set of permutations on $[1..n]$ with the restriction that the first $n-m$ symbols are in sorted order, for a given $m < n$.
(The first $n-m$ symbols in sorted order means they can be replaced by zeroes, although the strings with zeroes are not permutations.)
A set $A\subseteq S_{n,m}$ 
with Kendall-$\tau$ distance $d$ is called a $(n,m,d)$-PA or $(n,m,d)$-array. 
Let $P(n,m,d)$ denote the maximum cardinality of any $(n,m,d)$-array $A$.

In \cite{Jiang} Theorem \ref{Jiang} was proved using the set $\{1, d+1, 2d+1,\dots, \ceil{\frac{n+1}{d}}{d+1}\}$, which is a $(n+1,1,d)$-array. 
In general, a $(n,m,d)$-array can be much larger than one obtained by an iterative use of Theorem \ref{Jiang}. 
We exhibit a $(14,2,11)$-array with 5 permutations $\tau_1,\dots,\tau_5$ in Table \ref{t:14-2-11}.

\begin{table}[]
\centering
\setlength{\tabcolsep}{3pt}
\begin{tabular}{|c|rrrrrrrrrrrrrr|c|}
 \hline
  & 1 & 2 & 3 & 4 & 5 & 6 & 7 & 8 & 9 & 10 & 11 & 12 & 13 & 14 & $P_{\tau_i}(14,11)$\\
\hline
$\tau_1$ & 0 & 0 & 0 & 0 & 0 & 0 & 13 & 14 & 0 & 0 & 0 & 0 & 0 & 0 & 47,851\\
$\tau_2$  & 0 & 0 & 14 & 0 & 0 & 0 & 0 & 0 & 0 & 0 & 0 & 0 & 0 & 13 & 36,250\\
$\tau_3$  & 0 & 13 & 0 & 0 & 0 & 0 & 0 & 0 & 0 & 0 & 0 & 0 & 0 & 14 & 19,227\\
$\tau_4$  & 13 & 14 & 0 & 0 & 0 & 0 & 0 & 0 & 0 & 0 & 0 & 0 & 0 & 0 & 19,227\\
$\tau_5$  & 0 & 0 & 0 & 0 & 0 & 0 & 0 & 0 & 0 & 0 & 0 & 0 & 14 & 13 & 19,227\\
\hline
\end{tabular}
\vspace{5pt}
\caption{$(14,2,11)$-array with 5 permutations $\tau_1,\dots,\tau_5$. Since the first 12 symbols in all $\tau_i$ are sorted, they are replaced by zeros.
The last column contains lower bounds for $P_{\tau_i}(14,11), i=1,\dots,5$.}
\label{t:14-2-11}
\end{table}

For any integer $n > 1$ let $\pi_1(a,s)$ denote the string $n \ldots n-1$, and let $\pi_2(a,s)$ denote the string $n-1 \ldots n$, where '$a$' is the number of dots between the two symbols $n$ and $n-1$ and '$s$' is the position of the first symbol of the pair. We call the number '$a$' the {\it separation number} and the number '$s$' {\it the starting position}. 
For example, consider the set of permutations $A_{9,2,3}$, for $n=9,~ m=2$ and $d=3$, where the symbols in $[1..7]$ are in order and replace any blank positions or dots:

\vspace{-5pt}
\begin{itemize}
\setlength\itemsep{1mm}
    \item 

$\pi_1(0,s), s \in \{1, 3, 5, 7\}$,

\item 
$\pi_2(0,s), s \in \{2, 4, 6, 8\}$,

\item 
$\pi_1(3,s)$ and $\pi_2(3,s), s \in \{1, 4\}$,

\item 
$\pi_1(4,s)$ and $\pi_2(4,s), s \in \{2\}$,

\item 
$\pi_1(7,s)$ and $\pi_2(7,s), s \in \{1\}$.
\end{itemize}

To illustrate our notation, the set $A_{9,2,3}$ is shown in Table \ref{P923} (where, for convenience, the symbols in $[1..7]$, which appear in order, are replaced with -'s). The $(9,2,3)$-array $A_{9,2,3}$ has sixteen elements, so $P(9,2,3) \ge 16$. (Our search program does not find any larger $(9,2,3)$-array, so this may be optimal.)

\begin{table}[]
    \centering
    \begin{tabular}{|c|}
    
9 8 - - - - - - -\\
- 8 9 - - - - - -\\
- - 9 8 - - - - -\\
- - - 8 9 - - - -\\
- - - - 9 8 - - -\\
- - - - - 8 9 - -\\
- - - - - - 9 8 -\\
- - - - - - - 8 9\\
9 - - - 8 - - - -\\
8 - - - 9 - - - -\\
- 9 - - - - 8 - -\\
- 8 - - - - 9 - -\\
- - - 9 - - - 8 -\\
- - - 8 - - - 9 -\\
9 - - - - - - - 8\\
8 - - - - - - - 9 \\
\end{tabular}
\caption{A $(9,2,3)$-array $A_{9,2,3}$.}
    \label{P923}
\end{table}

Similarly, we give a $(10,2,3)$-array $A_{10,2,3}$ with elements:

\begin{itemize}

\item 
$\pi_1(0,s), s \in \{1, 3, 5, 7, 9 \}$,

\item 
$\pi_2(0,s), s \in \{2, 4, 6, 8\}$,

\item 
$\pi_1(3,s)$ and $\pi_2(3,s), ~s \in \{1, 3, 6\}$,

\item
$\pi_1(4,s)$ and $\pi_2(4,s), s \in \{ 4\}$.

\item 
$\pi_1(6,s)$ and $\pi_2(6,s), s \in \{ 1 \}$.

\item 
$\pi_1(7,s)$ and $\pi_2(7,s), s \in \{ 2\}$.
\end{itemize}

The $(10,2,3)$-array $A_{10,2,3}$ has twenty-one elements, so $P(10,2,3) \ge 21$. (Our search program again does not find any larger $(10,2,3)$-array.)

Consider, permutations $\pi_1(0,s)$ and $\pi_2(0,s+1)$ (or, respectively,  $\pi_2(0,s)$ and $\pi_1(0,s+1)$), with all symbols other than $n-1$ and $n$ in order.
The Kendall-$\tau$ distance between them is three because one needs two adjacent transpositions to move the symbol $n$ (respectively, $n-1$) from position $s$ to position $s+2$ and one needs one additional adjacent transpositions to move the symbol $n-1$, (respectively $n$)  in position $s+1$, where it occurs initially, back to position $s+1$ after its displacement by the movement of the symbol $n$ (respectively, $n-1$). 
Clearly, $d(\pi_1(r,u),\pi _1(s,t))\ge 3$ if $|r-s| \ge 3$, as it takes at least 3 adjacent transpositions to change the separation distance $r$ in one case to the separation distance $s$ in the other. 
Also, $d(\pi_1(r,u),\pi _1(s,t))\ge 3$ if $s\ge r+1$ and $t\ge u+1$.
As an example of the latter, note that $\pi_1(7,1)$ and $\pi_1(8,2)$ have distance three, as it takes one transposition to move $n$ in position 1 to position 2 and two more adjacent transpositions to move $n-1$ from position 8 to position 10.

In this way, one can verify  that $A_{9,2,3}$ and $A_{10,2,3}$ have Kendall-$\tau$ distance three.

Generally, define:
\vspace{-5pt}
\begin{itemize}
\setlength\itemsep{1mm}
    \item 
$D_{n,2,3}=
\{~ \pi_1(0,s)~ |~ s \in \{1,3, ... ,n-1 \} ~\}$, if $n$ is even, 
\item
$D_{n,2,3} =
\{ ~ \pi_1(0,s) ~ | ~ s \in \{1,3, ... ,n-2 \}~ \}$, if $n$ is odd, 
\item 
$E_{n,2,3} =
\{ ~ \pi_2(0,s) ~ | ~ s \in \{2,4, ... ,n-2 \} ~ \} ~ \}$, if $n$ is even,
\item
$E_{n,2,3} =
\{ ~ \pi_2(0,s) ~ | ~ s \in \{2,4, ... ,n-1\}$, if $n$ is odd, 
\item
$F_{n,2,3} = \{ ~ \pi_1(a,s), \pi_2(a,s) ~|~ s \in \{1,4,7, ...\} , a \in \{3,7,11,... \} ~\} $,
\item
$G_{n,2,3} = \{ ~ \pi_1(b,t), \pi_2(b,t) ~|~ t \in \{2,5,8,...\}, b \in \{4,8,12, ... \} ~ \}$.
\end{itemize}
\vspace{-5pt}

(The values of {\it s,t,a,b} are such that the resulting permutation has length $n$). 

Let $B_{n,2,3} ~=~D_{n,2,3} \bigcup E_{n,2,3}
\bigcup F_{n,2,3}
\bigcup G_{n,2,3}$. Observe that $B_{n,2,3}$ has pairwise Kendall-$\tau$ distance 3.
This follows from the discussion in the previous paragraph, when starting positions and separation numbers are close; otherwise, note that 
either separation numbers or starting positions 
differ by at least 3. (As illustrated in Table \ref{P923}.)

\begin{theorem}
\label{P_n_2_3}
For all $n\ge 5$,
$P(n,2,3)~ \ge 
\begin{cases}
    
\frac{(n+1)^2}{6}, & \text {if $n \equiv 5$ mod 6,}\\

\frac{n^2+2n-6}{6} & \text {if  $n \equiv 0$ or $n \equiv 4$ mod 6, }\\

\frac{n^2+2n-3}{6} & \text {if  $n \equiv 1$ or $n \equiv 3$ mod 6,}\\

\frac{n^2+2n-2}{6}
& \text {if $n \equiv 2$ mod 6.}\\
\end{cases}
$
\end{theorem}

\begin{proof} 
We show that $B_{n,2,3}$ has the given sizes. 
Clearly, $|D_{n,2,3}|+|E_{n,2,3}|=n-1$.
We assume that initially degree two polynomials can be used to describe each case, find the polynomials using Lagrange's formula, and then verify the correctness of the computed polynomial for other values. 
That is, we obtain the unique degree 2 polynomial passing through three given points. 
The results are dependent on, for which $i$, $n \equiv i$ mod 6.
For instance, when
$n \equiv 5$ mod 6, we have $|B_{17,2,3}|=54$,
$|B_{29,2,3}|=150$, and $|B_{41,2,3}|=294$. This yields the unique degree 2 polynomial: $\frac{(n+1)^2}{6}$. The polynomial is verified to be the correct size of $B_{n,2,3}$, for other values, when $n \equiv 5$ mod 6. Other cases are done similarly.
\end{proof}

Observe that $|A_{10,2,3}| = 21$, whereas $|B_{10,2,3}|=19$.
This is due to the fact that, $\pi_1(3,s)$ and $\pi_2(3,s)$, in $A_{10,2,3}$, are chosen for $s \in \{1,3,6\}$, not for $s \in \{1,4,7\}$.

Note that a two-fold iterative use of Theorem \ref{Jiang} gives

$P(n,2,3)~ \ge 
\begin{cases}
    
\frac{n^2}{9}, & \text {if $n \equiv 0$ mod 3,}\\

\frac{n^2+n-2}{9} & \text {if $n \equiv 1$ mod 3, and}\\

\frac{n^2+2n+1}{9} & \text {if $n \equiv 2$ mod 3}.\\

\end{cases}
$

\bigskip

So, Theorem \ref{P_n_2_3} yields a better result $P(n,2,3)\ge n^2/6+n/3-1$, for all $n \ge 5$. 
We show $P(n,2,3)\le n^2/5+5.8n-12$ in Theorem \ref{uP_n_2_3}.

\bigskip

Consider a two-fold iterative use of Theorem \ref{Jiang} for $n=14$ and $d=4$. The result is shown in Figure 
\ref{2fold}.

\begin{figure}[htb]
\centering
\begin{tabular}{|c|c|c|c|c|c|c|c|c|c|c|c|c|c|}
14&13&-&-&-&-&-&-&-&-&-&-&-&-  \\
13&-&-&-&14&-&-&-&-&-&-&-&-&-\\
13&-&-&-&-&-&-&-&14&-&-&-&-&-\\
13&-&-&-&-&-&-&-&-&-&-&-&14&-\\
14&-&-&-&-&13&-&-&-&-&-&-&-&-\\
-&-&-&-&14&13&-&-&-&-&-&-&-&-\\
-&-&-&-&13&-&-&-&14&-&-&-&-&-\\
-&-&-&-&13&-&-&-&-&-&-&-&14&-\\     
14&-&-&-&-&-&-&-&-&13&-&-&-&-\\
-&-&-&-&14&-&-&-&-&13&-&-&-&-\\
-&-&-&-&-&-&-&-&14&13&-&-&-&-\\
-&-&-&-&-&-&-&-&13&-&-&-&14&-\\
14&-&-&-&-&-&-&-&-&-&-&-&-&13\\
-&-&-&-&14&-&-&-&-&-&-&-&-&13\\
-&-&-&-&-&-&-&-&14&-&-&-&-&13\\
-&-&-&-&-&-&-&-&-&-&-&-&14&13\\  
\end{tabular}
\vspace{5pt}
\caption{Sixteen permutations in $S_{14,2}$ obtained by the proof of Theorem \ref{Jiang}, given in \cite{Jiang}, with "-" representing all missing elements (they are in order).
    }
\label{2fold}
\end{figure}
Observe that additional permutations in $S_{14,2}$ can be added to those shown in Figure \ref{2fold} and create a (14,2,4)-array with 25 elements. 
The additional ones are shown in Figure \ref{additional}.

\begin{figure}[htb]
\centering
\begin{tabular}{|c|c|c|c|c|c|c|c|c|c|c|c|c|c|}
-&-&14&13&-&-&-&-&-&-&-&-&-&-  \\
-&-&-&13&-&-&14&-&-&-&-&-&-&-\\
-&-&-&13&-&-&-&-&-&-&14&-&-&-\\
-&-&14&-&-&-&-&13&-&-&-&-&-&-\\
-&-&-&-&-&-&14&13&-&-&-&-&-&-\\
-&-&-&-&-&-&-&13&-&-&14&-&-&-\\

-&-&14&-&-&-&-&-&-&-&-&13&-&-\\
-&-&-&-&-&-&14&-&-&-&-&13&-&-\\
-&-&-&-&-&-&-&-&-&-&14&13&-&-\\     
\end{tabular}
\vspace{5pt}
\caption{Nine additional permutations in $S_{14,2}$ with "-" representing all missing elements.}
\label{additional}
\end{figure}

The additional elements shown in Figure \ref{additional} are at distance at least four from all other elements in Figure \ref{2fold} and 
\ref{additional}. 
For example, the first permutation in Figure \ref{additional} is at distance four from the first one in Figure \ref{2fold}, as both symbols, 14 and 13, need to be moved two positions. 
Similarly, the first permutation in Figure \ref{additional} is at distance four from the second one in Figure \ref{2fold}, as the symbol 13 needs to be moved three positions and the symbol 14 needs to be moved one. 
Also, the first two permutations in Figure \ref{additional} are at distance five, as the symbol 14 needs to move four positions, and to restore the symbol 13 to its original position, one more adjacent transposition is needed.

Note that the additional permutations in Figure \ref{additional} have a similar structure to the ones shown in Figure \ref{2fold}.
The difference is that the permutations in Figure \ref{additional} are positioned in the "gaps" between elements in those shown in Figure \ref{2fold}. Let $S_{n+2,2,4}$ denote the set of all permutations given by the two-fold use of Theorem \ref{Jiang}, as illustrated in Figure \ref{2fold} for $n=14$, together with the "additional" permutations, as illustrated in Figure \ref{additional} for $n=14$, which are all permutations with the symbol '$n$' in all positions $\equiv 2$ mod 4 and the symbol '$n-1$' in all positions $\equiv 3$ mod 4.

This example is the basis for the following theorem:

\begin{theorem} 
\label{improved}
For all $n \ge 4$, 
$P(n +2, 2,4) \ge
\begin{cases}
\frac{n(n+4)}{8} +1, & \text {if $n \equiv 0$ mod 4},\\

 \frac{(n+3)(n+1)}{8}, & \text {if $n \equiv 1$ or $n \equiv 3$ mod 4,} \\

\frac{(n+2)^2}{8},& \text {if $n \equiv 2$ mod 4},\\
\end{cases}$
\end{theorem}

\begin{proof}
    Consider the first two cases. The proof for the other cases follows in a similar manner. In each case, we calculate the size of $S_{n+2,2,4}$.

     For the first case, notice that
$\frac{n(n+4)}{8} +1 = (\frac{n+4}{4})^2 +
 (\frac{n}{4})^2$. The term ($ \frac{n+4}{4} )^2$ comes from a two-fold use of Theorem \ref{Jiang}, and it is equal to $\ceil{
    \frac{n+1}{4}} \cdot \ceil{
    \frac{n+2}{4}}$, when $n \equiv 0$ mod 4. The second term $(\frac{n}{4})^2$ comes from the insertion of the additional elements, as exhibited in Table \ref{additional}. As shown in our discussion of the example, the entire set of permutations thus formed has pairwise Kendall-$\tau$ distance 4.

    For the second case, notice that $\frac{(n+3)(n+1)}{8} = ( \frac{n+3}{4})^2 +( \frac{n-1}{4})^2+\frac{n-1}{4}$.
     The first term, $( \frac{n+3}{4}
     )^2$, comes from a two-fold use of Theorem \ref{Jiang}, and is equal to $\ceil{
    \frac{n+1}{4}} \cdot \ceil{\frac{n+2}{4}}$, when $n \equiv 1$ mod 4. 
    The last two terms come from the insertion of the additional elements, as exhibited in Table \ref{additional}. 
    Observe that when $n \equiv 1$ mod 4, there is an extra position at the end from those used in the two-fold use of Theorem \ref{Jiang} and this allows for $\frac{n-1}{4}$ additional permutations than without the extra position. 
    As shown in our discussion of the example, the entire set of permutations thus formed has pairwise Kendall-$\tau$ distance 4.
\end{proof}

Our search program does not find better results for $P(n+2,2,4)$, for any $n \le 250$. 
In fact, the search program finds a number for $P(n+2,2,4)$ that is exactly the same as those given in Theorem \ref{improved}. 
So, Theorem \ref{improved} may be optimum. A general theorem for the use of $P(n,m,d)$ is given below:

\begin{theorem}
\label{gen}
For any $m < n$ and d, $P(n,d) \ge P(n,m,d)\cdot P(n-m,d)$.
\end{theorem}

\begin{proof} 
Let $A$ be a $(n,m,d)$-array and $B$ be a $(n-m,d)$-array.
For each permutation $\pi$ in $A$ and each permutation $\tau$ in $B$, form the permutation $(\pi,\tau)$ by substituting the $n-m$ symbols in the order given by $\tau$ for the first $n-m$ symbols, given in order, in $\pi$. 

It is easily seen that $d((\pi,\tau), (\rho,\sigma)) \ge d$, if either $\pi \ne \rho$ or $\sigma \ne \tau$. That is, for $\pi, \rho \in A$, if $ \pi \ne \rho$, then $d(\pi,\rho) \ge d$.
Clearly, changing the order of the other $n-m$ symbols, which appear in order in permutations in $A$, does not make the distance smaller. 
A symmetric argument applies when $\sigma, \tau$ are different permutations in the $(n-m,d)$-array $B$.
\end{proof}

Using Theorem \ref{gen}, for example, we obtain $P(14,11) \ge 5\cdot P(12,11) \ge 5 \cdot 19,277 = 96,135$ which is better than obtained by Theorem \ref{Jiang}. 
The recursion implicit in Theorem \ref{gen} is also useful for computing $P(n,d)$ for large $n$. 
Consider, for example, $n=18$. 
There are 18! = 6,402,373,705,728,000 permutations on eighteen symbols, so it is infeasible to compute, say, $P(18,15)$ directly. 
However, the task becomes easier by dividing it into two parts, namely computing separately $P(18,8,15)$ and $P(10,15)$. That is, computing $P(10,15)$ involves looking at $10! = 3,628,800$ permutations and computing $P(18,8,15)$
involves looking at $\binom{18}{8} \cdot 8! = 43,758 \cdot 40,320 = 1,764,323,000$ permutations.

Moreover, one can improve on Theorem \ref{gen} using a modification. 
That is, given a $(n,m,d)$-array A, for each  $\tau$ in $A$, find the set of permutations of the remaining $n-m$ symbols with the $m$ symbols of $\tau$ fixed in their positions. 
The $m$ fixed symbols of $\tau$ can make different arrangements of the remaining $n-m$ symbols be at larger distance than they would be without the fixed symbols. 
Let $P_\tau(n,d)$ denote the maximum cardinality of any $(n,d)$-PA with the (largest) $m$ symbols fixed in the positions they occur in $\tau$, but where the other $n - m$ symbols can be in any order.  
Alternatively, we denote this quantity by $P(n,d;i_1,\dots,i_m)$, where $i_1,\dots,i_m$ are the fixed positions of symbols $n-m+1,\dots,n$, not necessarily in that order.
For example, in Table \ref{t:14-2-11}, for the arrangement $\tau_1$, there are $47,851$ permutations in $P_{\tau_1}(14,11)$ by filling in the symbols of $[1..12]$, whereas our bound for $P(12,11)$ is $19,227$. 
That is, the fixed symbols $13$ and $14$ in positions $7$ and $8$, respectively, allow more arrangements of the additional twelve symbols to be at distance 11.

\begin{theorem}
\label{better}
For any $(n,m,d)$-array $A$, $P(n,d) \ge \sum_{~\tau \in A} P_\tau(n,d)$. 
\end{theorem}

\begin{proof}
Let A be a $(n,m,d)$-array and, for each permutation $\pi \in A$, let $\tau$  be a permutation in an $(n,d)$-PA with the highest $m$ symbols in the same position as in $\pi$. Form the new permutation $(\pi,\tau)$ by substituting the $n-m$ symbols in the order given by $\tau$ for the first $n-m$ symbols, given in order, in $\pi$. 

It is easily seen, as in the proof of Theorem \ref{gen}, that $d((\pi,\tau), (\rho,\sigma)) \ge d$, if either $\pi \ne \rho$ or $\sigma \ne \tau$.
\end{proof}

For example, we saw the result $P(14,11) \ge 96,125$ using Theorem \ref{gen}, with a $(14,2,11)$-array with five permutations $\tau_i,i=1,\dots,5$ as shown in Table \ref{t:14-2-11}. 
We computed lower bounds for $P_{\tau_i}(14,11)$, as shown in the last column of this table.
By Theorem \ref{better}, we obtain the improved lower bound of $P(14,11) \ge
\sum_{i=1}^5 P_{\tau_i}(14,11) \ge 141,782$.

Since the greedy approach with randomness was successful in computing permutation arrays, we adapted it for computing permutation arrays using Theorems \ref{gen} and \ref{better}.
There are two computational problems in applying Theorem \ref{better}:

{\bf Problem IndexPA}. List in lexicographic order permutations $\pi$ in $S_{n,m}$. 

{\bf Problem FixedPA}. Given a permutation $\tau$, list in lexicographic order permutations in $S_n$  such that $m$ largest numbers fixed in the positions they occur in $\tau$.

Recall that the standard algorithm for listing permutations in $S_n$ in lexicographic order is as follows. 
The first permutation is $\pi=(1,2,\dots,n)$. 
If $\pi$ is the current permutation, then 
(i) find the largest $i$ with $\pi(i)<\pi(i+1)$ (stop if $i$ is not found),   
(ii) swap $\pi(i)$ and the smallest $\pi(j)>\pi(i), j>i$, and
(iii) reverse $\pi(i+1),\dots,\pi(n)$. 

Surprisingly, this algorithm can be applied to solve problem IndexPA if\\
(i) we start with sequence $s=(0,\dots,0,n-m+1,\dots,n)\in\Z^n$ and \\
(ii) when $s$ is computed, we output a permutation where zeros $s$ are replaced by $1,2,\dots,n-m$ in this order.

To solve problem FixedPA, we can assume that $1,2,\dots,n-m$ are in positions $k_1<k_2<\dots<k_{n-m}$ in $\tau$, \ie $\tau(k_i)=i, i\in [1..n-m]$.
Problem FixedPA can be solved by listing  permutations $\pi$ in $S_{n-m}$ in lexicographic order and reporting corresponding permutations $\pi'$ defined as 

\[\pi'(i)=
\begin{cases}
\tau(i) & \text {if $\tau(i)>n-m$},\\
\pi(\tau(i))  & \text {otherwise}.\\
\end{cases}
\]

\section{Lower and upper bounds for $P(n,m,d)$ and $P(n,d)$.}

Similar to Theorem \ref{Jiang}, we give a theorem about recursively computing $P(n,m+1,d)$. 
The proof is nearly identical to the proof of Theorem \ref{Jiang}, but now involves $n-m$ instead of $n$.

\begin{theorem}
\label{Pnmd}
For any $m < n$ and d, $P(n,m+1,d) \ge \left\lceil \frac{n-m}{d}\right\rceil \cdot P(n,m,d)$.
\end{theorem}

\begin{proof}
Similar to the proof of Theorem \ref{Jiang}   
one can place, into any permutation $\sigma$ of a $(n,m,d)$-array $A$, the new symbol "$n-m$"  and replace any one of the following: the first zero, the $d+1$-st zero, the $2d+1$-th zero, etc. 
As the Kendall-$\tau$ distance is at least $d$ between any two of these placements, the resulting new permutations are at distance at least $d$. 
\end{proof}

There are $\frac{n!}{(n-m)!}$ permutations in $S_{n,m}$ for finding $P(n,m,d)$.  
When $m$ is small, this 
is relatively small compared to the $n!$ permutations to explore for finding $P(n,d)$.
Also, $P(n,m,d)$ generalizes $P(n,d)$ as $P(n,d)=P(n,n,d)$.
Finding exact values or bounds for $P(n,m,d)$ is an interesting problem in its own right. 
Clearly, $P(n,1,d)=\ceil{n/d}$.
In general, by Theorem \ref{Pnmd}
\begin{equation}
\label{simple}
P(n,m,d)\ge 
\left\lceil \frac{n-m+1}{d}\right\rceil
\cdot
\left\lceil \frac{n-m+2}{d}\right\rceil
\cdot\dots\cdot
\left\lceil \frac{n}{d}\right\rceil.
\end{equation}

Note that the lower bound (3) can be improved in a manner similar to that done for $P(n,d)$.
Computed values of $P(n,m,d)$ are shown in Tables \ref{Pnm8-9}, \ref{Pnm10-11}, \ref{Pnm12-13}, and \ref{Pnm14-15}. Others can be seen in Tables \ref{Just1} and \ref{Just2}.

We denote by $\id$ the identity permutation $(1,2,\dots,n)$.

\begin{prop}
\label{Pnmd2}
$P(n,m,d)\ge 2$ if $d\le 
mn-m(m+1)/2$. The bound for $d$ is tight for all $n>m\ge 1$.
\end{prop}

\begin{proof}
Let $\pi=(n,n-1,\dots,n-m+1,1,2\dots,n-m)$. 
The bubble sort for $\pi$ uses $n-1$ transpositions for symbol $n$, $n-2$ transpositions for symbol $n-1$, etc. Then   
$
d(\id,\pi)=(n-1)+(n-2)+\dots +(n-m)
=nm-(1+2+\dots+m)=mn-m(m+1)/2$.

The bound is tight since for any permutation $\sigma\ne \pi$, $d(\id,\sigma)<mn-m(m+1)/2$.
\end{proof}

We improve the bound in Equation \ref{simple} for $m=2$. 

\begin{theorem}
\label{Pn2d}
For any $d\ge 1$,
(a) $P(n,2,d)\ge 3$, if $d\le n+\floor{n/3}-2$, and
(b) $P(n,2,d)\ge 5$ if $d\le n-2$. 
\end{theorem}

\begin{proof}
(a) Let $\tau_1=(n-1,n,1,2,\dots,n-2), \tau_2=(1,\dots,x-1,n-1,x,\dots,n-2,n)$
and $\tau_3=(1,\dots,x,n,x+1,\dots,n-1)$ where $x=\floor{n/3}$, see an example in Table \ref{t:9-2-10}.
Transformation of $\tau_1$ to $\tau_2$ requires $n-1$ transpositions for symbol $n-1$ and $x-1$ transpositions for symbol $n$.
Then $d(\tau_1,\tau_2)=n+x-2\ge d$. 
Similarly $d(\tau_1,\tau_3)=(n-2)+x\ge d$, and
$d(\tau_2,\tau_3)=(n-x)+(n-x-2)=2n-2x-2\ge n+x-2\ge d$.

{
\begin{table}[htb]
\centering
\setlength{\tabcolsep}{3pt}
\begin{tabular}{|c|rrrrrrrrr|}
 \hline
  & 1 & 2 & 3 & 4 & 5 & 6 & 7 & 8 & 9 \\
\hline
$\tau_1$ & 8 & 9 & 1 & 2 & 3 & 4 & 5 & 6 & 7\\
$\tau_2$ & 1 & 2 & 9 & 3 & 4 & 5 & 6 & 7 & 8\\
$\tau_3$ & 1 & 2 & 3 & 8 & 4 & 5 & 6 & 7 & 9\\
\hline
\end{tabular}
\vspace{5pt}
\caption{$P(9,2,10)\ge 3$.}
\label{t:9-2-10}
\end{table}
}

(b) 
Suppose $n=2k$. Consider 5 permutations $\tau_i,i=1,\dots,5$ where symbols $n-1$ and $n$ are placed at positions 
1 and 2 for $\tau_1$,
$n-1$ and $n$ for $\tau_2$,
$k$ and $k+1$ for $\tau_3$,
$1$ and $n$ for $\tau_4$,
$n$ and $1$ for $\tau_5$, see an example in Table \ref{t:12-2-10}.
We show that $d(\tau_i,\tau_j)\ge n-2$ if $1\le i<j\le 5$. 
For all pairs $i,j\in \{1,2,4,5\}$ with $i<j$, transformation of $\tau_i$ to $\tau_j$ 
requires $n-2$ transpositions for only one of two symbols $n-1$ or $n$.
Transformation of $\tau_3$ to any $\tau_i,i=1,2,4$ requires $k-1$ transpositions for each symbol $n-1$ and $n$.
Transformation of $\tau_3$ to any $\tau_5$ requires $k-1$ transpositions for each symbol $n-1$ and $n$ after transposition of $n-1$ and $n$.

Similarly, a $(n,2,n-2)$-array can be constructed for $n=2k+1$ where symbols $n$ and $n-1$ are placed at positions 
$k$ and $k+1$ for $\tau_3$, see an example in Table \ref{t:13-2-11}.
\end{proof}

{ 
\begin{table}[htb]
\centering
\setlength{\tabcolsep}{3pt}
\begin{tabular}{|c| r r rrrrrrrrrr|}
 \hline
  & 1 & 2 & \ 3 & \  4 & \ 5 & 6 & 7 & \ 8 & \ 9 & 10 & 11 & 12\\
\hline
$\tau_1$ & 11 & 12 & 0 & 0 & 0 & 0 & 0 & 0 & 0 & 0 & 0 & 0\\
$\tau_2$ & 0 & 0 & 0 & 0 & 0 & 0 & 0 & 0 & 0 & 0 & 11 & 12\\
$\tau_3$ & 0 & 0 & 0 & 0 & 0 & 11 & 12 & 0 & 0 & 0 & 0 & 0\\
$\tau_4$ & 11 & 0 & 0 & 0 & 0 & 0 & 0 & 0 & 0 & 0 & 0 & 12\\
$\tau_5$ & 12 & 0 & 0 & 0 & 0 & 0 & 0 & 0 & 0 & 0 & 0 & 11\\
\hline
\end{tabular}
\vspace{5pt}
\caption{$P(12,2,10)\ge 5$. The first 10 symbols in all $\tau_i$ are in the sorted order and replaced by zeros.
}
\label{t:12-2-10}
\end{table}
}

{ 
\begin{table}[htb]
\centering
\setlength{\tabcolsep}{3pt}
\begin{tabular}{|c| r r rrrrrrrrrrr|}
 \hline
  & 1 & 2 & \ 3 & \  4 & \ 5 & 6 & 7 & \ 8 & \ 9 & 10 & 11 & 12 & 13\\
\hline
$\tau_1$ & 12 & 13 & 0 & 0 & 0 & 0 & 0 & 0 & 0 & 0 & 0 & 0 & 0\\
$\tau_2$ & 0 & 0 & 0 & 0 & 0 & 0 & 0 & 0 & 0 & 0 & 0 & 12 & 13\\
$\tau_3$ & 0 & 0 & 0 & 0 & 0 & 13 & 12 & 0 & 0 & 0 & 0 & 0 & 0\\
$\tau_4$ & 12 & 0 & 0 & 0 & 0 & 0 & 0 & 0 & 0 & 0 & 0 & 0 & 13\\
$\tau_5$ & 13 & 0 & 0 & 0 & 0 & 0 & 0 & 0 & 0 & 0 & 0 & 0 & 12\\
\hline
\end{tabular}
\vspace{5pt}
\caption{An example for $P(13,2,11)\ge 5$.}
\label{t:13-2-11}
\end{table}
}

We have constructed a program for computing $P(n,m,d)$ for various values of $n,m,$ and $d$. For each of the $\binom{n}{m}$ positions for $m$ symbols out of $n$, and each of the possible $m!$ orders of the $m$ symbols, the program uses the random/Greedy strategy described earlier. 
That is, it chooses a specified number of random choices first and then tries adding all remaining possible permutations in increasing order.
When $m$ is small, the program finds solutions quickly. It allows one to compute $P(15,12)$, for example,  without examining all 15! permutations of 15 symbols. That is, by Theorem \ref{gen} one can first compute, for example, $P(15,6,12)$, which as shown in Table \ref{Pnm12-13} is at least 622, and then compute $P(12,12)$.

We prove an upper bound for $P(n,2,3)$. 

\begin{theorem}
\label{uP_n_2_3}
For all $n\ge 4$,
$P(n,2,3)\le n^2/5+5.8n-12$. 
\end{theorem}

\begin{proof} 
There are $n^2-n$ permutations in $S_{n,2}$. 
We assume $n\ge 4$.
Let $R_1$ be the set of permutations $\pi\in S_{n,2}$ such that $\pi(1)\in\{n-1,n\}$ or $\pi(n)\in\{n-1,n\}$. 
Let $R_2$ be the set of permutations $\pi\in S_{n,2}$ such that $|\pi^{-1}(n-1)-\pi^{-1}(n)|=1$.\footnote{$\pi^{-1}(n-1)$ and $\pi^{-1}(n)$ are the positions of $n-1$ and $n$ in $\pi$, respectively.}
Let $R=R_1\cup R_2$ and $R_{12}=R_1\cap R_2$.
Then
\[
|R|=|R_1|+|R_2|-|R_{12}|=(4n-6)+(2n-2)-4=6n-12.
\]

For a permutation $\pi\in S_{n,2}$, let $B_1(\pi)$ be the set of permutations $\sigma\in S_{n,2}$ such that $d(\pi,\sigma)\le 1$.
If $\pi\in S_{n,2}\setminus R$ then $|B_1(\pi)|=5$.
Let $A$ be an $(n,2,3)$-array. 
Then the balls of radius one $B_1(\pi), \pi\in A$ are disjoint. 
The union of balls $B_1(\pi), \pi\in A\setminus R$ has size $5\cdot |A\setminus R|\le n^2-n$.
Therefore $|A|\le (n^2-n)/5+|R|=n^2/5+5.8n-12$.
\end{proof}

The upper bound of Theorem \ref{uP_n_2_3} is the Gilbert-Varshamov type bound for $d=3$. 
It can be generalized to any odd $d\ge 3$ and any $m\ge 2$. 
Let $V_{m,k}=|\{x\in \Z^m~|~ \sum_{i=1}^m |x_i|\le m\}$. 

\begin{theorem}
\label{uP_n_m_d}
For all $m\ge 2$ and $k\ge 1$,
$P(n,m,2k+1)\le n^m/V_{m,k}+O(n^{m-1})$. 
\end{theorem}

\begin{proof} 
We assume that $m$ and $k$ are constants.
Let $d=2k+1$.
There are $n(n-1)\dots (n-m+1)$ permutations in $S_{n,m}$. 
Every permutation $\pi$ in $S_{n,m}$ can be specified as $m$-tuple $p=(p_1,p_2,\dots,p_m)$ such that $\pi(p_1)=n-m+1,\dots,\pi(p_m)=n$, i.e. $p_1,p_2,\dots,p_m$ are the positions of symbols $n-m+1,\dots,n$, respectively.  
Let $R$ be the set of $m$-tuples $p$ such that $|p_i-p_j|\le d$ for some two positions $p_i$ and $p_j$, $1\le i<j\le m$. 
Then $|R|=O(n^{m-1})$.

For a permutation $\pi\in S_{n,m}$, let $B_k(\pi)$ be the set of permutations $\sigma\in S_{n,m}$ such that $d(\pi,\sigma)\le k$.
If $\pi\in S_{n,m}\setminus R$ then $|B_k(\pi)|=V_{m,k}$.
Let $A$ be an $(n,m,d)$-array. 
Then the balls of radius $k$ $B_k(\pi), \pi\in A$ are disjoint. 
The union of balls $B_k(\pi), \pi\in A\setminus R$ has size $V_{m,k}\cdot |A\setminus R|\le |S_{n,m}|$.
Therefore $|A|\le |S_{n,m}|/V_{m,k}+|R|=n^m/V_{m,k}+O(n^{m-1})$.
\end{proof}

One can show that $V_{m,1}=2m+1$ for $m\ge 2$. Then
\begin{equation}
\label{uP_n_m_3}
P(n,m,3)\le n^2/(2m+1)+O(n^{m-1}).    
\end{equation}
Since $V_{2,2}=|\{(0,\pm 2),(\pm 2,0), (a,b) ~:~ |a|,|b|\le 1\} |=13$, 
\begin{equation}
\label{uP_n_m_3}
P(n,2,5)\le n^2/13+O(n).    
\end{equation}

As shown in Tables \ref{better2a} and \ref{better2b}, Theorems \ref{gen} and \ref{better}
 are useful for obtaining improved lower bounds for $P(n,d)$ when the Kendall-$\tau$ distance $d$ is close to $n$. Some of the improvements are substantial. For example, we obtain $8,413,437$ as a lower bound for $P(16,11)$, whereas Theorem \ref{Ge} gives a lower bound of $1,700,585$. We also give lower bounds for $P(n,m,d)$, for $8 \le d \le 15$ and $10 \le n \le 20$ in Tables \ref{Pnm8-9}, \ref{Pnm10-11},
\ref{Pnm12-13}, and \ref{Pnm14-15}, which are given in the appendix. 


Many of the improved lower bounds given in Tables \ref{better2a} and \ref{better2b} are explained in Tables \ref{Just1} and \ref{Just2}.
Some computations took a week (or more) on Apple MacBook Air computers with an M1 or M2 processor.
Most of the results are obtained by Theorem \ref{better} using a computation of $P(n,m,d)$. There are many choices for the value of $m$. Most of our results are with $m=8$. 

In Tables \ref{better2a} and \ref{better2b}, "previous" results are given by Theorem \ref{Ge}, with the use of Theorems \ref{Even}, \ref{Jiang2}, and \ref{Jiang}, when appropriate. 
Tables \ref{Just1} and \ref{Just2} give the methods we used to obtain improved lower bounds.

\begin{table}[htb]
\setlength{\tabcolsep}{2pt}
    \centering
    \begin{tabular}{|c|c|c|c|c|c|c|c|c|c|c|}
    \hline
    n:d &5&6&7&8&9&10&11\\ 
    \hline    \hline
    12& {\bf 899,809}&{\bf 595,160}&{\bf 129,298}&{\bf 85,091}&{\bf 73,105}&see Table 2&see Table 2\\
    (previous)&(720,304)&(360,152)&(46,741)&(23,371)&(3,305)&&\\
    \hline
    13&-&-&{\bf 629,301}&{\bf 520,253}&{\bf 236,764}&{\bf 158,208}&{\bf 51,046}\\ 
    (previous)&(9,363,942)&(4,681,971)&(607,632)&(303,816)&(42,962)&(21,481)&(3,196)\\
    \hline
    14 & - & - & {\bf 6,522,803} & {\bf 3,693,995} & {\bf 930,601} & {\bf 571,415} & {\bf 177,098}\\
(previous) & Theorem 2* & Theorem 2* & (5,232,791)&(2,616,396)&(313,063)&(156,532)&(19,704)\\
\hline
    
    15&-&-&-&-&{\bf 6,846,611}&{\bf 3,878,969}&{\bf 1,182,803}\\

    (previous)&Theorem 2*&Theorem 2*&Theorem 2*&(39,245,930)&(4,695,943)&(2,347,972)&(295,549)\\
    \hline
    16&-&-&-&-&-&{\bf 30,193,558}&{\bf 8,413,437}
    \\

    (previous)&Theorem 2*&Theorem 2*&Theorem 2*&Theorem 2*&(33,259,910)&(16,629,955)&(1,700,585)\\
    \hline
    
    17&-&-&-&-&-&-&{\bf 66,863,784}\\
    
    (previous)&Theorem 2*&Theorem 2*&Theorem 2*&Theorem 2*&Theorem 2*&Theorem 2*&(28,909,942)
    \\
    \hline
    18&-&-&-&-&-&-&-\\
    (previous)&Theorem 2*&Theorem 2*&Theorem 2*&Theorem 2*&Theorem 2*&Theorem 2*&(520,378,955)\\
    \hline
    \end{tabular}
\vspace{5pt}
    \caption{Improved lower bounds for $P(n,d)$ using Theorem \ref{better}, for $5 \le d \le 11$. Previous results from "Theorem 2*", meaning Theorem \ref{Ge}, with appropriate use of Theorems \ref{Even}, \ref{Jiang2}, and \ref{Jiang}. 
    }
    \label{better2a}
\end{table}

\begin{table}[htb]
\setlength{\tabcolsep}{2pt}
    \centering
    \begin{tabular}{|c|c|c|c|c|c|c|c|}
    \hline
    n:d &12&13&14&15&16&17&18\\ 
    \hline    \hline
    12&see Table 2&see Table 2&see Table 2&see Table 2&see Table 2&see Table 2&{\bf 418}\\
    (previous)&&&&&&&(1)\\
    \hline
    13&{\bf 29,859}&{\bf 14,158}&{\bf 10,756}&{\bf 5,527}&{\bf 4,322}&{\bf 1,024}&{\bf 771}\\  
    (previous)&(1,598)&(246)&(123)&(20)&(10)&(2)&(1)\\
    \hline
    14&{\bf 112,338}&{\bf 55,730}&{\bf 41,673}&{\bf 15,674}&{\bf 8,941}&{\bf 4,429}&{\bf 3,190}\\
    (previous)&(9,852)&(1,283)&(642)&(86)&(43)&(6)&(3)\\
    \hline
    15&{\bf 706,114}&{\bf 190,218}&{\bf 159,967}&{\bf 66,194}&{\bf 44,416}&{\bf 20,842}&{\bf 14,610}\\
    (previous)&(147,775)&(19,237)&(9,619)&(1,283)&(642)&(88)&(44)\\
    \hline
    16&{\bf 4,977,819}&{\bf 1,665,481}&{\bf 1,043,093}&{\bf 394,158}&{\bf 259,662}&{\bf 111,714}&{\bf 77,044}\\
(previous)&(850,293)&(89,935)&(44,968)&(4,872)&(2,436)&(269)&(135)\\
    \hline
    17&{\bf 38,745,418}&{\bf 12,013,922}&{\bf 7,398,247}&{\bf 2,657,379}&{\bf 1,706,757}&{\bf 687,795}&{\bf 462,163}\\
    (previous)&(14,454,970)&(1,528,892)&(764,446)&(82,813)&(41,407)&(4,567)&(2,284)\\
    \hline
    18&-& {\bf 96,452,048}&{\bf 57,732,698}&{\bf 19,618,333}&{\bf 12,411,066}&{\bf 4,671,851}&{\bf 3,099,772}\\
    (previous)&(260,189,477)&(27,520,040)&(13,760,020)&(1,490,629)&(745,315)&(82,206)&(41,103)\\
    \hline
    19&-&-&-&{\bf 39,236,666}&{\bf 24,822,132}&{\bf 9,343,702}&{\bf 6,199,544}\\
    (previous)&Theorem 2*
    &Theorem 2*&Theorem 2*&(18,600,815)&(9,300,408)&(965,441)&(482,721)\\
    \hline
    \end{tabular}
\vspace{5pt}
    \caption{Improved lower bounds for $P(n,d)$ using Theorem \ref{better}, for $12 \le d \le 18$. Previous results from "Theorem 2*", meaning Theorem \ref{Ge}, with appropriate use of Theorems \ref{Even}, \ref{Jiang2}, and \ref{Jiang}. }
    \label{better2b}
\end{table}

\begin{table}[htb]
\setlength{\tabcolsep}{2pt}
    \centering
    \begin{tabular}{|c|c|c|c|c|c|c|c|}
    \hline
    n:d &5&6&7&8&9&10&11\\ 
    \hline    \hline
    12&$\displaystyle \sum_{i  \in A} P(n,5;i)$&$m=8$
    &$\displaystyle
    \sum_{i \in A} P(n,7;i)$
    &$\displaystyle \sum_{i \in A} P(n,8;i)$&
    $\displaystyle \sum_{i \in A} P(n,9;i)$&-&-\\
    &A=\{2,7,12\}&|A|=81,954&$A=\{3,10\}$&$A=\{3,11\}$&$A=\{3,12\}$&&\\
    \hline
    13&&&$m=8$&$m=8$&$\displaystyle
    \sum_{i \in A}P(n,9;i)$&$\displaystyle \sum_{i \in A} P(n,d;i)$&$\displaystyle \sum_{i \in A} P(n,d;i)$\\
    &&&$|A|=12,604$&$|A|=81,954$&$A=\{3,12\}$&$A=\{2,12\}$&$A=\{2,13\}$\\
    \hline
    14&&&&$m=8$&$m=8$&$m=8$&$m=8$\\  
    &&&&$|A|=14,779$&$|A|=26,300$&$|A|=18,620|$&$|A|=7,909$\\
    \hline
    15&&&&&$m=8$&$m=8$&$m=8$\\
    &&&&&$|A|=50,126$&$|A|=35,264$&$|A|=14,715$\\
    \hline
    16&&&&&&$m=8$&$m=8$\\
    &&&&&&$|A|=63,538$&$|A|=26,075$\\
    \hline
    17&&&&&&&$m=8$\\
    &&&&&&&$|A|=44,489$\\
    
    \hline
    \end{tabular}
\vspace{5pt}
    \caption{Methods used to obtain lower bounds. An entry "$m=8$", for example,  means the lower bound was obtained computing an $(n,8,d)$-array A and then computing $\displaystyle \sum_{\tau  \in A} P_{\tau}(n,d)$. }
    \label{Just1}
\end{table}

\begin{table}[htb]
\setlength{\tabcolsep}{2pt}
    \centering
    \begin{tabular}{|c|c|c|c|c|c|c|c|}
    \hline
    $n$:$d$ &12&13&14&15&16&17&18\\ 
    \hline  \hline
    13&$P(n,d;7)$&$P(n,d;7)$&$P(n,d;7)$&$P(n,d;7)$&$P(n,d;7)$&$m=8$&$m=8$\\
    &&&&&&$|A|=$ 316&$|A|=$ 258\\
    \hline
    14&$m=8$&
$\Sigma\{P(n,d;i,j)\}$ &$\Sigma\{P(n,d;i,j)\}$&$\Sigma
\{P(n,d;i,j)\}$&$m=8$&$m=8$&$m=8$\\ 
    &&
    $|\{(i,j) \in A\}|$ &
    $|\{(i,j) \in A\}|$ &
    $|\{(i,j) \in A\}$ &&&\\
    &$|A|=$ 5,880&$A=\{(3,7),$&$A=\{(1,3),$&$A=\{(2,3),$&$|A|=962$&$|A|=576$&$|A|=472$\\
    && $(13,14),(2,12)\}$&$(4,14),(6,11)\}$&$(6,14),(14,6)\}$&&&\\
  \hline
  15 & $m=8$ & $m=7$ & $m=8$ & $m=8$ & $m=8$ & $m=8$ & $m=8$\\  
    &$|A|=$ 10,860&$|A|=$ 1,382&$|A|=$ 4,017&$|A|=$ 2,165&$|A|=$ 1,719&$|A|=$ 1,012&$|A|=$ 821\\
    \hline
    16&$m=8$&$m=8$&$m=8$&$m=8$&$m=8$&$m=8$&$m=8$\\
    &$|A|=19,215$&$|A|=9,139$&$|A|=6,983$&$|A|=3,705$&$|A|= 2,926$&$|A|=1,708$&$|A|=1,385$\\
    \hline
    17& $m=8$ & $m=8$ & $m=8$ & $m=8$&$m=8$ & $m=8$ & $m=8$\\
    &$|A|=$ 32,647&$|A|=$ 15,287&$|A|=$ 11,647&$|A|=$ 6,134&$|A|=$ 4,841&$|A|=$ 2,776&$|A|=$ 2,249\\
    \hline
    18&&$m=8$&$m=8$&$m=8$&$m=8$&$m=8$&$m=8$\\
    &&$|A|=$ 24,912&$|A|=$ 18,914&$|A|=$ 9,856&$|A|=$ 7,722&$|A|=$ 4,413&$|A|=$ 3,551\\
    
    \hline
    \end{tabular}
\vspace{5pt}
    \caption{Methods used to obtain lower bounds. An entry "$m=8$", for example,  means the lower bound was obtained computing an $(n,8,d)$-array A and then computing $\displaystyle \sum_{\tau  \in A} P_{\tau}(n,d)$. }
    \label{Just2}
\end{table}

\section{Automorphism Lower Bounds}

It is known that for a permutation $\pi(x):\F_q \to\F_q$, where $\F_q$ denotes a finite field of order $q$, the operations of multiplying by a non-zero constant $a$, adding a constant $c$, and adding to the argument a constant $b$, each yields another permutation on $\F_q$. 
This is a well-known equivalence relation on permutation polynomials. That is, $a\pi(x+b)+c$, for all non-zero $a$ and all $b,c \in \F_q$, is again a permutation. We use this to search for sets of permutations at specified Kendall-$\tau$ distance $d$.
That is, the search can be done for a set of representative permutations and expanded into a full set of permutations using operations on the representatives.
Our program verifies that the full set of permutations has the stipulated Kendall-$\tau$ distance.

\newpage

\begin{itemize}
    \item 
Use the operation $\pi(x)+c$ on the 17 representatives shown in Table \ref{6a3}. This gives 102 permutations for $P(6,3)$.

\item
Use the operations $a\pi(x)+c$ on the 14 representatives given in Table \ref{9a7}. 
This gives $1,008$ permutations for $P(9,7)$.

\item
Use the operations $a\pi(x)+c$ on the 8 representatives given in Table \ref{9a8}. 
This gives 576 permutations for $P(9,8)$.

\item
Use the operations $a\pi(x)+c$ on the four representatives given in Table \ref{9a9}.
This gives $288$ permutations for $P(9,9)$.

\item
Use the operations $\pi(x)+c$ on the 12 representatives given in Figure \ref{7a6}.
This gives $84$ permutations for $P(7,6)$.

\item 
Use the operation $a\pi(x)+c$ on the 8 representatives given in Figure \ref{8a6}.
This gives $448$ permutations for $P(8,6)$. 

\item 
Use the operation $a\pi(x)+c$ on 67 representatives given in Figure \ref{8a3}. This gives 3,752 permutations for $P(8,3)$.

\item 
Use the operation $a\pi(x)+c$ on 12 representatives given  in Table \ref{8a5}. This gives 672 permutations for $P(8,5)$.

\item
Use the operation $a\pi(x)+c$ on 40 representatives given in Table \ref{8a4} representatives.  This gives 2,242 permutations for $P(8,4)$.

\item
Use the operation $a\pi(x)+c$ on 3 representatives given in Table \ref{8a7}. 
This gives $168$ permutations for $P(8,7)$. 

\item
Use the operation $\pi(x)+c$ on the 48 permutations given in Table \ref{7a4}. This gives $336$ permutations for $P(7,4)$.

\item
Use the operation $\pi(x)+c$ on the 18 permutations given in Table \ref{7a5}. This gives $126$ permutations for $P(7,5)$.

\end{itemize}

\newpage

\section{Conclusions and Open Questions}

Theorems \ref{gen} and \ref{better} give many improved lower bounds. Tables  \ref{Automorphism}, \ref{Random}, \ref{better2a}, and \ref{better2b} give improvements on previous results.
As previously stated in Section \ref{sec:intro}, lower bounds obtained by our recursive technique can be much larger than those given by Theorem \ref{Ge}.

Our work on good patterns for $(n,m,d)$-arrays continues. 
We conjecture that $(n,m,d)$-arrays can be used to compute other improved lower bounds for $P(n,d)$.
Another interesting direction for future research is upper bounds for $P(n,m,d)$. 
We conjecture that some lower bounds in Tables \ref{Pnm8-9}, \ref{Pnm10-11}, \ref{Pnm12-13}, and \ref{Pnm14-15} are tight.
An interesting open problem is the asymptotic behavior of $P(n,2,3)$. 
If $P(n,2,3)/n^2$ tends to a constant $c$, what is the value of $c$? By Theorems \ref{P_n_2_3} and \ref{uP_n_2_3}, $\frac 16\le c\le \frac 15$. 


\begin{thebibliography}{10}

\bibitem{Abdollahi22}
A.~Abdollahi, J.~Bagherian, F.~Jafari, M.~Khatami, F.~Parvaresh, and
  R.~Sobhani.
\newblock New bounds on the size of permutation codes with minimum {K}endall
  $\tau$-distance of three.
\newblock {\em arXiv}, abs/2206.10193, 2022.

\bibitem{bm-cpec-10}
A.~Barg and A.~Mazumdar.
\newblock Codes in permutations and error correction for rank modulation.
\newblock {\em IEEE Transactions on Information Theory}, 56(7):3158--3165,
  2010.

\bibitem{bls-18}
S.~Bereg, A.~Levy, and I.~H. Sudborough.
\newblock Constructing permutation arrays from groups.
\newblock {\em Designs, Codes and Cryptography}, 86(5):1095--1111, 2018.

\bibitem{bmmms-19}
S.~Bereg, Z.~Miller, L.~G. Mojica, L.~Morales, and I.~H. Sudborough.
\newblock New lower bounds for permutation arrays using contraction.
\newblock {\em Designs, Codes and Cryptography}, 87:2105--2128, 2019.

\bibitem{buzaglo}
S.~Buzaglo and T.~Etzion.
\newblock Bounds on the size of permutation codes with the {K}endall tau
  metric.
\newblock {\em IEEE Trans. on Inform. Theory}, 61(6):3241--3250, 2015.

\bibitem{chu2004}
W.~Chu, C.~J. Colbourn, and P.~Dukes.
\newblock Constructions for permutation codes in powerline communications.
\newblock {\em Designs, Codes and Cryptography}, 2004.

\bibitem{Deza2}
M.~M. Deza and T.~Huang.
\newblock Metrics on permutations, a survey.
\newblock {\em J. Comb. Inf. System Sci.}, 23:173--185, 1998.

\bibitem{Jiang}
A.~Jiang, M.~Schwartz, and J.~Bruck.
\newblock Correcting charge-constrained errors in the rank-modulation scheme.
\newblock {\em IEEE Transactions on Information Theory}, 56(5):2112--2120,
  2010.

\bibitem{Klove10}
T.~Kl{\o}ve, T.-T. Lin, S.-C. Tsai, and W.-G. Tzeng.
\newblock Permutation arrays under the {C}hebyshev distance.
\newblock {\em IEEE Trans. on Info. Theory}, 56(6):2611 -- 2617, 2010.

\bibitem{Vijayakumaran}
S.~Vijayakumaran.
\newblock Largest permutation codes with {K}endall $\tau$-metric in {$S_4$} and
  {$S_5$}.
\newblock {\em IEEE Communications Letters}, 20(10):1912--1915, 2016.

\bibitem{wang17}
X.~Wang, Y.~Zhang, Y.~Yang, and G.~Ge.
\newblock New bounds of permutation codes under {H}amming metric and
  {K}endall's {\(\tau\)}-metric.
\newblock {\em Des. Codes Cryptography}, 85(3):533--545, 2017.

\end{thebibliography}


\bigskip

\bigskip

\begin{table}[htb]
\centering
\begin{tabular}
{ccc}
\begin{tabular}{|c|r|c|c|c|c|}
\hline
n:m& 2& 3&4&5&6\\ 
\hline    \hline
10 & 5&14&37&113&335 \\ 
11& 5 &16&55&186&645 \\ 
12 & 6 &21&73&285& 1145\\ 
13& 6&26&99&428& 1920 \\
14& 8&31 &130&625&3117\\
15 & 8 & 37 &172&884&4872\\
16 & 10 &45&219&1233&7367 \\
17 & 10 & 52 & 278 & 1676 & 10828\\
18 & 13 & 61 & 344 & 2227 & 15567\\
19 & 13 & 71 & 426 & 2939 & 21862\\
20 & 15 & 80 & 517 & 3805 & 30196\\
\hline
\end{tabular} 

&

\begin{tabular}{|c|r|c|c|c|c|}
\hline
 n:m& 2& 3&4&5&6\\ 
\hline    \hline
10 & 3&9&24&63&162 \\ 
11& 5 &15&34&99&301 \\ 
12 & 5 &16&46&149& 523\\ 
13& 6&18&59&219&861 \\
14& 6& 22&78&315&1383 \\
15 & 7 &26 &100&445&2119\\
16 &8&31&128&610&3165\\
17&8&36&162&824&4613\\
18&10&42&201&1097&6589\\
19&10&49&244&1427&9179\\
20 & 12 & 55 & 292 & 1827 & 12581\\
\hline
\end{tabular}
\end{tabular}
\vspace{5pt}
\caption{Lower bounds for $P(n,m,8)$ (left) and $P(n,m,9)$  (right).}
\label{Pnm8-9}
\end{table} 

{\bf Appendix}

\begin{table}[htb]
\centering
\begin{tabular}
{ccc}
\begin{tabular}{|c|r|c|c|c|c|}
\hline
$n:m$ & 2& 3&4&5&6\\ 
\hline    \hline
10 & 3&7&19&48& 125\\ 
11& 5 &10&27&76&226 \\ 
12 & 5 &13&37&116&394 \\ 
13& 6&16&50&167& 644\\
14& 6&18 &64&241&1011 \\
15 & 6 &21&83&342&1570\\
16 &6&25&103&467&2337\\
17&8&30&129&629&2239\\
18&8&35&158&829&3185\\
19&10&40&192&1084&4405\\
20&10&46&233&4184&6017\\
\hline
\end{tabular}
&
\begin{tabular}{|c|r|c|c|c|c|}
\hline
$n:m$ & 2& 3&4&5&6\\ 
\hline    \hline
10 &3&6&13&27& 73\\ 
11& 3 &7&16&41& 128\\ 
12 & 3 &10&22&61& 214\\ 
13& 5 & 11 & 31 & 96 & 344\\
14& 5 & 13 &37&120&539 \\
15 &5 &17&55&163&810\\
16 &6&20&70&220&1193\\
17&6&23&86&366&1716\\
18&7&26&106&472&2413\\
19&8&31&127&618&3362\\
20&8&35&151&789&4571\\
\hline
\end{tabular}
\end{tabular}
\vspace{5pt}
\caption{Lower bounds for $P(n,m,10)$ (left)  and $P(n,m,11)$ (right).}
\label{Pnm10-11}
\end{table}

\begin{table}[htb]
\centering
\begin{tabular}{ccc}
\begin{tabular}{|c|r|c|c|c|c|}
\hline
$n:m$ & 2& 3&4&5&6\\ 
\hline    \hline
10 &2 &6&13&26&58\\ 
11& 3 &7&17&40&101 \\ 
12 & 3 &9&23&59&168\\ 
13&3 &10&30&84&273\\
14&5 &13&37&117&420 \\
15 &5&16&45&159&622\\
16 &5&17&58&216&919\\
17&6&20&72&287&1323\\
18&6&22&87&375&1859\\
19&6&25&103&485&2580\\
20&8&30&125&620&3503\\
\hline
\end{tabular}

&
\begin{tabular}{|c|r|c|c|c|c|}
\hline
$n:m$ & 2& 3&4&5&6\\ 
\hline    \hline
10 &2 &4&10&20&37\\ 
11& 2 &6&13&28&63 \\ 
12 & 3 &7&16&40&103\\ 
13&3 &9&22&56&163\\
14&3 &10&27&79& 247\\
15 &5&12&35&106&370\\
16 &5&15&44&141&533\\
17&5&16&52&181&757\\
18&6&18&63&242&1058\\
19&6&20&73&308&1447\\
20&6&23&90&390&1965\\
\hline
\end{tabular}
\end{tabular}
\vspace{5pt}
\caption{Lower bounds for $P(n,m,12)$ (left) and $P(n,m,13)$ (right).}
\label{Pnm12-13}
\end{table} 

\begin{table}[htb]
\centering
\begin{tabular}{ccc}
\begin{tabular}{|c|r|c|c|c|c|}
\hline
$n:m$ & 2& 3&4&5&6\\ 
\hline    \hline
10 &2 &4&10&16&30\\ 
11& 2 &4&11&23& 51\\ 
12 & 3 &6&15&34&85\\ 
13& 3&7&18&48&133\\
14&3 &9&24&65& 203\\
15 &3&10&30&88&298\\
16 &5&13&38&118&431\\
17&5&15&46&153&609\\
18&5&16&54&197&844\\
19&6&18&63&254&1163\\
20&6&20&75&323&1568\\
\hline
\end{tabular}

&
\begin{tabular}{|c|r|c|c|c|c|}
\hline
$n:m$ & 2& 3&4&5&6\\ 
\hline    \hline
10 & 2&4&6&12&19\\ 
11& 2 &4&10&20& 31\\ 
12 & 2 &5&12&21&48\\ 
13& 3&6&15&30&72\\
14&3 &7&16&40&107\\
15 &3&9&23&52&154\\
16 &3&10&29&84&221\\
17&5&12&35&109&385\\
18&5&14&41&138
&530\\
19&5&16&41&174&720\\
20&5&17&46&220&961\\
\hline
\end{tabular}
\end{tabular}
\vspace{5pt}
\caption{Lower bounds for $P(n,m,14)$ (left) and $P(n,m,15)$ (right).}
\label{Pnm14-15}
\end{table}


\begin{table}[htb]
\setlength{\tabcolsep}{2pt}
\centering
\begin{tabular}{|c|c|c|c|}
\hline
0 1 2 4 8 3 7 5 6
&
0 1 2 7 8 5 3 4 6
&
0 1 3 4 7 2 8 6 5
&
0 1 3 8 2 6 7 4 5
\\
\hline
0 1 3 8 4 6 5 7 2
&
0 1 4 5 6 7 3 8 2
&
0 1 4 5 8 2 7 6 3
&
0 1 6 2 3 4 7 8 5
\\
\hline
0 1 6 2 8 7 5 4 3
&
0 1 6 4 5 2 3 8 7
&
0 1 6 7 3 4 8 5 2
&
0 1 7 2 4 6 8 5 3
\\
\hline
0 1 7 4 8 3 5 2 6
&
0 1 8 5 7 4 6 3 2&&\\
\hline
\end{tabular}
\vspace{5pt}
\caption{Representatives for P(9,7).}
\label{9a7}
\end{table}
\bigskip

\begin{table}[htb]
\setlength{\tabcolsep}{2pt}
\centering
\begin{tabular}{|c|c|c|c|c|c|}
\hline
0 1 2 3 5 4&
0 1 2 4 5 3&
0 1 3 5 4 2&
0 1 5 4 2 3&
0 2 3 4 1 5&
0 2 4 5 1 3\\
\hline
0 2 5 3 4 1&
0 3 1 4 2 5&
0 3 2 5 1 4&
0 3 4 2 5 1&
0 3 5 4 1 2&
0 4 1 5 3 2\\
\hline
0 4 2 1 3 5&
0 4 5 3 2 1&
0 5 2 1 3 4&
0 5 3 1 2 4&
0 5 4 2 1 3&\\
\hline

\end{tabular}
\vspace{5pt}
\caption{Representitives for P(6,3),}
\label{6a3}
\end{table}
\bigskip

\begin{table}[htb]
\centering
\begin{tabular}{|c|c|c|c|}
\hline
0 1 2 6 5 8 7 4 3&
0 1 3 8 4 5 2 6 7&
0 1 4 6 5 3 7 2 8&
0 1 5 2 4 7 3 6 8\\
\hline

\end{tabular}
\caption{Representatives for P(9,9).}
\label{9a9}
\end{table}
\bigskip

\begin{table}[!htb]
\centering
\begin{tabular}{|c|c|c|c|}
\hline
0 1 2 3 8 4 6 5 7&
0 1 2 5 8 6 3 7 4&
0 1 4 5 2 8 6 7 3&
0 1 5 3 2 4 6 8 7\\
\hline
0 1 5 6 4 8 3 7 2&
0 1 6 4 7 2 5 8 3&
0 1 6 7 3 2 8 5 4&
0 1 8 3 6 5 7 2 4\\
\hline

\end{tabular} 

\caption{Representatives for P(9,8).}
\label{9a8}

\end{table}
\bigskip

\begin{table}[!htb]
\centering
\begin{tabular}{|c|c|c|c|c|}
\hline
0 1 3 6 5 4 2&
0 1 4 2 3 6 5&
0 1 6 2 5 4 3&
0 2 3 4 1 5 6&
0 2 3 6 5 1 4\\
\hline
0 3 4 6 1 2 5&
0 3 5 4 1 2 6&
0 4 5 6 3 1 2&
0 5 2 4 3 6 1&
0 5 3 6 1 2 4\\
\hline
0 6 3 5 4 2 1&
0 6 4 2 1 3 5&&&\\
\hline
\end{tabular}
\caption{Representatives for P(7,6).}
\label{7a6}
\end{table}
\bigskip

\begin{table}[!htb]
\centering
\begin{tabular}{|r|r|r|r|}
\hline
0 1 7 4 5 6 2 3&
0 2 1 5 3 4 6 7&
0 2 6 4 7 3 1 5&
0 3 7 5 4 2 1 6\\
\hline
0 5 4 6 7 1 2 3&
0 7 3 1 2 6 5 4&
0 7 5 4 3 6 1 2&
0 7 6 4 2 1 3 5\\
\hline
\end{tabular}
\caption{Representatives for P(8,6).}
\label{8a6}
\end{table}

\begin{table}[!htb]
\centering
\begin{tabular}{|r|r|r|r|r|}
\hline
0 2 3 6 5 4 7 1&
0 2 4 3 1 5 6 7&
0 3 2 1 6 4 7 5&
0 3 5 1 6 2 7 4&
0 5 7 2 4 6 1 3\\
\hline
0 6 3 4 5 2 1 7&
0 6 3 7 1 5 2 4&
0 6 5 4 7 3 1 2&
0 7 1 5 4 6 2 3&
0 7 3 6 4 2 1 5\\
\hline
0 7 4 1 2 6 5 3&
0 7 5 6 4 1 3 2&&&\\
\hline

\end{tabular}
\caption{Representatives for P(8,5).}
\label{8a5}
\end{table}

\begin{table}[!htb]
\centering
\begin{tabular}{|r|r|r|r|r|}
\hline
0 1 2 3 4 5 6 7&
0 1 2 5 3 6 7 4&
0 1 3 5 7 2 6 4&
0 1 5 4 3 6 2 7&
0 1 6 2 7 3 4 5\\
\hline
0 1 6 3 4 2 7 5&
0 1 6 7 4 5 2 3&
0 1 7 3 2 5 6 4&
0 1 7 5 3 2 4 6&
0 1 7 5 6 3 4 2\\
\hline
0 2 3 5 1 4 7 6&
0 2 3 5 7 6 4 1&
0 2 3 6 5 4 7 1&
0 2 4 1 6 5 7 3&
0 2 4 5 6 3 1 7\\
\hline
0 2 5 1 7 4 3 6&
0 2 5 3 4 6 7 1&
0 2 5 4 3 1 6 7&
0 2 5 6 4 1 7 3&
0 2 6 4 3 5 1 7\\
\hline
0 2 6 4 7 1 5 3&
0 3 1 5 4 7 2 6&
0 3 2 4 1 7 6 5&
0 3 2 5 4 7 1 6&
0 3 2 6 1 4 5 7\\
\hline
0 3 6 2 4 5 1 7&
0 3 7 4 5 6 2 1&
0 3 7 5 4 2 1 6&
0 4 1 6 2 3 5 7&
0 4 2 7 3 1 5 6\\
\hline
0 4 2 7 5 6 1 3&
0 4 5 6 2 1 3 7&
0 4 6 1 7 2 3 5&
0 4 6 2 5 3 7 1&
0 4 6 2 7 1 5 3\\
\hline
0 4 7 5 2 3 1 6&
0 4 7 6 3 5 2 1&
0 5 1 6 7 4 3 2&
0 5 1 7 3 6 2 4&
0 5 2 1 6 3 7 4\\
\hline
0 5 2 3 6 4 1 7&
0 5 2 6 4 3 7 1&
0 5 3 1 4 6 2 7&
0 5 3 2 6 1 7 4&
0 5 3 4 1 2 7 6\\
\hline
0 5 3 7 6 1 4 2&
0 5 4 6 2 7 1 3&
0 5 4 6 3 1 2 7&
0 5 6 3 1 2 7 4&
0 5 6 3 7 4 1 2\\
\hline
0 5 7 6 4 3 1 2&
0 6 1 5 2 3 4 7&
0 6 2 4 3 7 5 1&
0 6 3 1 7 4 5 2&
0 6 3 7 2 4 5 1\\
\hline
0 6 4 3 5 7 1 2&
0 6 5 1 7 3 2 4&
0 6 7 1 3 5 4 2&
0 6 7 5 3 2 1 4&
0 7 1 2 3 4 5 6\\
\hline
0 7 1 3 5 4 6 2&
0 7 1 4 3 6 2 5&
0 7 3 4 2 1 5 6&
0 7 3 6 1 4 2 5&
0 7 4 6 3 1 2 5\\
\hline
0 7 4 6 5 2 3 1&
0 7 5 1 2 3 6 4&&&\\
\hline

\end{tabular}
\caption{Representatives for P(8,3).}
\label{8a3}
\end{table}

\begin{table}[!htb]
\centering
\begin{tabular}{|r|r|r|r|r|}
\hline
0 1 4 5 7 6 3 2&
0 1 7 3 2 5 6 4&
0 2 1 3 7 4 5 6&
0 2 1 5 7 4 6 3&
0 2 1 6 7 5 4 3\\
\hline
0 2 3 6 1 5 4 7&
0 2 4 3 5 6 1 7&
0 2 5 3 7 4 6 1&
0 2 7 1 4 5 3 6&
0 2 7 3 1 4 6 5\\
\hline
0 2 7 3 6 5 1 4&
0 2 7 6 1 4 5 3&
0 3 2 1 5 7 4 6&
0 3 5 6 4 7 1 2&
0 3 5 7 6 1 2 4\\
\hline
0 3 6 2 5 1 7 4&
0 4 1 6 2 3 5 7&
0 4 1 7 6 2 3 5&
0 4 2 1 5 6 3 7&
0 4 2 5 7 6 3 1\\
\hline
0 4 2 7 1 5 6 3&
0 4 3 1 7 5 6 2&
0 4 3 5 6 1 7 2&
0 5 2 1 6 3 7 4&
0 5 3 2 6 1 7 4\\
\hline
0 5 3 2 7 1 4 6&
0 5 4 2 1 3 6 7&
0 5 4 7 6 2 3 1&
0 5 6 2 1 7 4 3&
0 5 6 4 1 3 2 7\\
\hline
0 5 7 1 6 4 2 3&
0 5 7 3 2 4 6 1&
0 6 1 2 4 3 5 7&
0 6 7 2 4 3 1 5&
0 7 1 2 6 3 5 4\\
\hline
0 7 2 5 1 4 6 3&
0 7 2 5 3 6 4 1&
0 7 4 3 1 5 2 6&
0 7 5 1 4 2 3 6&
0 7 5 6 2 1 4 3\\
\hline
\end{tabular}
\caption{Representatives for P(8,4).}
\label{8a4}
\end{table}

\begin{table}[!htb]
\centering
\begin{tabular}{|r|r|r|}
\hline
0 5 3 1 4 6 2 7&
0 6 1 3 2 5 7 4&
0 7 3 1 2 6 5 4\\
\hline
\end{tabular}
\caption{Representatives for P(8,7).}
\label{8a7}
\end{table}

\begin{table}[!htb]
\centering
\begin{tabular}{|c|c|c|c|c|}
\hline
0	1	2	4	3	6	5&
0	1	2	5	4	6	3&
0	1	3	2	6	5	4&
0	1	3	4	2	5	6&
0	1	4	5	6	2	3\\
\hline
0	1	5	3	6	2	4&
0	1	6	2	3	4	5&
0	1	6	5	2	4	3&
0	2	1	3	5	4	6&
0	2	3	4	1	6	5\\
\hline
0	2	3	6	5	4	1&
0	2	4	1	5	3	6&
0	2	4	6	5	1	3&
0	2	5	1	6	3	4&
0	2	5	3	4	1	6\\
\hline
0	2	5	6	4	3	1&
0	2	6	1	4	3	5&
0	3	1	5	2	4	6&
0	3	1	6	5	4	2&
0	3	2	5	1	6	4\\
\hline
0	3	2	6	1	4	5&
0	3	4	1	6	2	5&
0	3	4	2	5	1	6&
0	3	4	5	6	1	2&
0	3	6	5	4	2	1\\
\hline
0	4	2	1	6	3	5&
0	4	2	5	3	6	1&
0	4	3	1	5	2	6&
0	4	3	6	2	5	1&
0	4	5	1	2	3	6\\
\hline
0	4	6	1	3	2	5&
0	4	6	5	1	2	3&
0	5	1	2	3	4	6&
0	5	2	4	1	6	3&
0	5	3	1	4	6	2\\
\hline
0	5	3	6	2	1	4&
0	5	4	1	6	3	2&
0	5	4	3	6	2	1&
0	5	6	1	4	2	3&
0	6	1	4	2	5	3\\
\hline
0	6	1	5	3	4	2&
0	6	2	5	3	1	4&
0	6	3	1	4	2	5&
0	6	3	2	4	5	1&
0	6	3	5	1	2	4\\
\hline
0	6	4	2	3	1	5&
0	6	4	5	3	2	1&
0	6	5	2	4	1	3&&\\
\hline
\end{tabular}
\caption{Representatives for P(7,4).}
\label{7a4}
\end{table}

\begin{table}[!htb]
\centering
\begin{tabular}{|c|c|c|c|c|}
\hline
0	1	4	2	5	3	6&
0	1	4	6	3	2	5&
0	1	5	2	6	4	3&
0	2	1	3	5	4	6&
0	2	4	5	6	3	1\\
\hline
0	2	6	4	1	3	5&
0	3	1	5	6	4	2&
0	3	2	4	5	1	6&
0	3	2	6	1	4	5&
0	3	5	4	6	2	1\\
\hline
0	4	3	1	5	2	6&
0	4	3	6	2	1	5&
0	4	5	1	6	3	2&
0	5	1	3	4	2	6&
0	5	3	2	6	1	4\\
\hline
0	6	1	2	5	3	4&
0	6	5	2	4	1	3&
0	6	5	3	4	1	2&&\\
\hline
\end{tabular}
\caption{Representatives for P(7,5).}
\label{7a5}
\end{table}


\end{document}